\documentclass[11pt]{amsart}
\usepackage{amsmath,amssymb,amsthm, hyperref, amscd}
\usepackage{enumitem}
\usepackage[T1]{fontenc}
\usepackage{textcomp}
\usepackage{tikz}
\usetikzlibrary{matrix,arrows,decorations.pathmorphing}
 \usepackage{tikz-cd}
\usepackage[scale=0.95,ttdefault]{AnonymousPro}
\hypersetup{
   colorlinks=true,
   linkcolor=blue,
    urlcolor=cyan,
    }

\usepackage[
    backend=biber,
    style=alphabetic,
    useprefix=true,
    maxnames=100,
    maxcitenames=100,
    sorting=anyt
]{biblatex}
\newcommand{\fm}{\mathfrak{m}}
\newcommand{\fn}{\mathfrak{n}}

\newcommand{\fp}{\mathfrak{p}}
\newcommand{\fq}{\mathfrak{q}}
\newcommand{\fP}{\mathfrak{P}}

\newcommand{\Spec}{\operatorname{Spec}}

\newcommand{\HT}{\operatorname{ht}}

\newcommand{\depth}{\operatorname{depth}}

\newcommand{\colim}{\operatorname{colim}}

\newcommand{\coker}{\operatorname{coker}}

\newcommand{\cO}{\mathcal{O}}
\newcommand{\cU}{\mathcal{U}}

\newcommand{\bA}{\mathbf{A}}
\newcommand{\bE}{\mathbf{E}}

\newcommand{\bP}{\mathbf{P}}

\newcommand{\bZ}{\mathbf{Z}}

\newcommand{\Tag}[1]{\href{https://stacks.math.columbia.edu/tag/#1}{\texttt{#1}}}
\newcommand{\citestacks}[1]{\cite[Tag \Tag{#1}]{stacks}}
\newcommand{\citetwostacks}[2]{\cite[Tags \Tag{#1} and \Tag{#2}]{stacks}}

\newtheorem{Thm}{Theorem}[section]

\newtheorem{Lem}[Thm]{Lemma}
\newtheorem{Cor}[Thm]{Corollary}

\theoremstyle{definition}
\newtheorem{Def}[Thm]{Definition}

\theoremstyle{remark}
\newtheorem{Fact}[Thm]{Fact}
\newtheorem{Rem}[Thm]{Remark}
\newtheorem{Clm}[Thm]{Claim}
\newtheorem{Ques}[Thm]{Question}

\title{Formal lifting of dualizing complexes and consequences}
\author{Shiji Lyu}
\address{Department of Mathematics, Statistics, and Computer Science\\University of Illinois at Chicago\\Chicago, IL
60607-7045\\USA}
\email{\href{mailto:slyu@uic.edu}{slyu@uic.edu}}
\urladdr{\url{https://homepages.math.uic.edu/~slyu/}}

\begin{document}

\begin{abstract}
    We show that for a Noetherian ring $A$ that is $I$-adically complete for an ideal $I$,
    if $A/I$ admits a dualizing complex, so does $A$.
    This gives an alternative proof of the fact that a Noetherian complete local ring admits a dualizing complex.
    We discuss several consequences of this result.
    We also consider a generalization of the notion of dualizing complexes to infinite-dimensional rings and prove the results in this generality.
    In addition, we give an alternative proof of the fact that every excellent Henselian local ring admits a dualizing complex,
    using ultrapower.
\end{abstract}

\maketitle

%

\section{Introduction}

The following problem has a long history in commutative algebra, cf. \cite[Remarque 7.4.8]{EGA4_2}.
\begin{Ques}\label{ques:WhatPhasLift}
    What properties $\bP$ of rings satisfy the \emph{lifting property},
    that is, for every Noetherian ring $A$ and ideal $I$ of $A$, if $A$ is $I$-adically complete and $A/I$ satisfies $\bP$,
    then $A$ satisfies $\bP$.
\end{Ques}

There have been numerous studies on the lifting property and its variants,
for many important properties $\bP$.
For example, lifting property holds for $\bP$=``Nagata'' \cite{Marot-Nagata-Lift} and $\bP$=``quasi-excellent'' \cite{formal-lifting-excellence-Gabber},  but not for $\bP$=``excellent'' or $\bP$=``universally catenary'' \cite{Greco-Universal-Catenary-No-Lift}.
We refer the reader to \cite[Appendix]{formal-lifting-excellence-Gabber} for more information.

One main objective of this article is to show that the property of admitting a dualizing complex satisfies the lifting property,
which, as fundamental as it may be, seems to be missing in the literature.
For applications, especially Theorem \ref{thm:etlocHasDual}, we consider the following generalization of dualizing complexes to infinite-dimensional rings.

\begin{Def}\label{def:weakDual0}
    Let $A$ be a Noetherian ring.
    Let $K\in D(A)$.
    We say $K$ is a \emph{pseudo-dualizing complex}
    if $K\in D^b_{Coh}(A)$ and $K_\fp$ is a dualizing complex for $A_\fp$ for all $\fp\in\Spec(A)$.
\end{Def}

We now state our main theorem.
Note that applying to the case $A$ local and $I$ maximal, we recover the existence of dualizing complexes for Noetherian complete local rings.
Therefore, this fact now admits a proof that does not use the Cohen structure theorem.

\begin{Thm}[=Theorem \ref{thm:GLOBALdeform}]\label{thm:GLOBALdeform0}
    Let $A$ be a Noetherian ring, $I$ an ideal of $A$.
    Assume that $A$ is $I$-adically complete. 
    If $A/I$ admits a pseudo-dualizing (resp. dualizing) complex $K_1$,
    then $A$ admits a pseudo-dualizing (resp. dualizing) complex $K$ such that $R\operatorname{Hom}_{A}(A/I,K)\cong K_1$.
\end{Thm}

This result implies, quite formally, certain openness results (Theorem \ref{thm:DualOpen} and Corollary \ref{cor:Gor2impliesopen}).
In turn, we characterize when a Noetherian scheme of dimension 1 admits a dualizing complex (Corollary \ref{cor:Dim1HasDual}),
and obtain that a quasi-excellent ring admits a pseudo-dualizing complex \'etale-locally (Theorem \ref{thm:etlocHasDual}),
via a known result of Hinich \cite{Hinich-hsHasDual} that an excellent Henselian local ring admits a dualizing complex.
We point out to the reader that the results in \S \ref{sec:Open} 
are new even for finite-dimensional rings to the best of our knowledge,
except where noted.

Another main objective of this article is to illustrate that Definition \ref{def:weakDual0} is, hopefully, 
a robust generalization of the classical concept of dualizing complexes.
We show that pseudo-dualizing complexes are preserved by upper shriek (Corollary \ref{cor:UpperShriek}),
induce a dualizing functor on $D^b_{Coh}$ (Corollary \ref{cor:DualizingFunctorGood}),
and are unique up to twist (Corollary \ref{cor:DualizeUniqueuptoL}).
The author believes it should be possible to develop the coherent duality theory for infinite-dimensional schemes using pseudo-dualizing complexes in place of dualizing complexes.
Moreover, Sharp's conjecture is true in this generality by the same proof, see Theorem \ref{thm:SharpKawasaki}.
One can also expect that the concept of fundamental dualizing complexes \cite{fundamental-dual} and related results can be extended to this generality.

To make our article accessible to a broader audience, we do not use $\bE_\infty$- or animated rings and the derived $\infty$-category,
although they helped a lot in the thought process.
The author hopes that many results in this article can be generalized to truncated animated rings and dgas.
We also produce an alternative proof of Hinich's result, 
avoiding the use of dgas in the original proof.
In exchange, we use ultrapowers of a local ring.
We show a flatness result (Theorem \ref{thm:MaptoUltraFlat}) which may be of independent interest.
It is a strengthening of \cite[Lemma 3.5]{Lyu-Splinter}.
We will not use Hinich's result except for Theorem \ref{thm:etlocHasDual}.\\

Our article is structured as follows. 
\S\ref{sec:DbCoh} is a short preparation. 
In \S\ref{sec:DualHs} we give an alternative proof of Hinich's result.
In \S\ref{sec:PesudoDual}, we introduce pseudo-dualizing complexes and show several desirable properties.
In \S\ref{sec:Lift}, we prove our main result on formal lifting.
In \S\ref{sec:Open} and \S\ref{sec:OneDim} we discuss consequences.
Finally, in \S\ref{sec:quot} we make some remarks on quotient of Cohen-Macaulay rings that do not rely on the material on pseudo-dualizing complexes.\\

We thank Rankeya Datta, Longke Tang, and Kevin Tucker for helpful discussions.
We thank Karl Schwede and Kevin Tucker for bringing the lifting problem of dualizing complexes to the author's attention.\\


We use the following facts about dualizing complexes for local rings without explicit reference.
They all follow from \citestacks{0A7U}, for $M=A$.
\begin{Fact}
    Let $A$ be a Noetherian local ring,
    $K$ a dualizing complex for $A$.
    Then $\dim A-\depth A=\max \{b-a\mid a\leq b, H^a(K)\neq 0, H^b(K)\neq 0\}$.
    If $a\in\bZ$ is such that $\dim\operatorname{Supp}(H^a(K))=\dim A$,
    then $K\in D^{\geq a}(A)$.
    $K$ is normalized if and only if the minimal $a$ so that $H^a(K)\neq 0$ is $-\dim A$.
\end{Fact}

\section{Bounded pseudo-coherent complexes}\label{sec:DbCoh}
\begin{Lem}\label{lem:0A6Aforiso}
    Let $A$ be a Noetherian ring. Let $K,M\in D^b_{Coh}(A)$.
    Let $S$ be a multiplicative subset of $A$.
    Assume $S^{-1}K\cong S^{-1}M$.
    Then for some $f\in S$,
    $K_f\cong M_f$.
\end{Lem}
\begin{proof}
    For any $X,Y\in D^b_{Coh}(A)$,
    we have \[S^{-1}\operatorname{Hom}_{D(A)}(X,Y)=\operatorname{Hom}_{D(S^{-1}A)}(S^{-1}X,S^{-1}Y)\] by \citestacks{0A6A}.
    Thus we can spread out a quasi-isomorphism and its inverse and the results are inverse quasi-isomophisms.
\end{proof}

\begin{Lem}\label{lem:DescendDbCoh}
    Let $(A_i)_i$ be a direct system of rings,
    $A=\colim_i A_i$.
    Assume that $A_i\to A$ is flat for all $i$.
    Let $K\in D^b(A)$.
    Assume $K$ is pseudo-coherent.
    Then $K\cong K_i^\bullet\otimes_{A_i}^L A$ for some $i$ 
    and some  finite complex $K_i^\bullet$ of finitely presented $A_i$-modules.

    If $a\leq b\in\bZ$ are given such that $K\in D^{[a,b]}(A)$,
    we can choose  $K_i^\bullet$
    so that $K_i^m=0$ for all $m\not\in [a,b]$ and $K_i^m$ free for all $a<m\leq b$.
\end{Lem}
\begin{proof}
    $K$ is represented by a bounded above complex $F^\bullet$ of finite free $A$-modules
    such that $F^m=0$ for all $m>b$.
    Let $K^\bullet$ be the complex of $A$-modules with
    $K^m=F^m\ (m> a), K^{a}=\coker(F^{a-1}\to F^a), K^m=0\ (m<a)$,
    that is, $K^\bullet$ is the canonical truncation $\tau_{\geq a}F^\bullet$.
    Then $K^\bullet$ also represents $K$,
    and $K^\bullet$ is a finite complex of finitely presented $A$-modules
    such that $K^m=0$ for all $m\not\in [a,b]$ and $K^m$ free for all $a<m\leq b$.
    By \citestacks{05N7}, there exists an index $i$ and a complex of finitely presented $A_i$-modules $K_i^\bullet$  
    such that $K_i^m=0$ for all $m\not\in [a,b]$ and $K_i^m$ free for all $a<m\leq b$,
    and that
    $K_i^\bullet\otimes_{A_i}A$ is isomorphic to $K^\bullet$ as cochain complexes.
    Since $A_i\to A$ is flat,  
    $K\cong K_i^\bullet\otimes_{A_i}^L A$.
\end{proof}

\section{Dualizing complexes for Henselian local rings}\label{sec:DualHs}

In this section we prove a generalization of the main result of \cite{Hinich-hsHasDual},
Theorem \ref{thm:HenselHasDual}.
We will only use the original result of \cite{Hinich-hsHasDual} for the application, 
Theorem \ref{thm:etlocHasDual}.

For ultraproducts of rings, first-order statements, and {\L}o\'s's Theorem,
see \cite[Chapter 2]{SchoutensBOOK}.

\begin{Lem}\label{lem:modabddImpliesbdd}
    Let $A$ be a Noetherian local ring or an ultraproduct of Noetherian local rings, $f\in A$ a noninvertible nonzerodivisor.
    Let $(E^\bullet,d^\bullet)$ 
    be a cochain complex of $A$-modules.
    If $E^\bullet/fE^\bullet$ is exact at some degree $m$
    and $E^{m-1},E^m,E^{m+1}$ are finite free,
    then $E^\bullet$ is exact at degree $m$.
\end{Lem}
\begin{proof}
    Choose a basis of $E^{m-1},E^m,E^{m+1}$,
    so $d^{m-1}$ and $d^m$ are represented by matrices with coefficients in $A$.
    The statement of the lemma is a first-order statement of the coefficients and $f$, 
    so it suffices to show the lemma for a Noetherian local $A$.

    The condition $E^\bullet/fE^\bullet$ is exact at degree $m$ implies $\ker(d^m)=\operatorname{im}(d^{m-1})+(fE^m\cap \ker(d^m))$.
    Since $f$ is a nonzerodivisor on $E^{m+1}$,
    $fE^m\cap \ker(d^m)=f\ker(d^m)$.
    Thus $\ker(d^m)=\operatorname{im}(d^{m-1})+f \ker(d^m)$.
    By Nakayama's Lemma,
    $\ker(d^m)=\operatorname{im}(d^{m-1})$.
\end{proof}

\begin{Thm}\label{thm:MaptoUltraFlat}
    Let $(A,\fm,k)$ be a Noetherian local ring, $(B,\fn,l)$ a Noetherian local flat $A$-algebra with $\fn=\fm B$.
    Let $A_\natural$ be an ultrapower of $A$.
    Then any $A$-algebra map $B\to A_\natural$, if exists, is faithfully flat.
\end{Thm}
\begin{proof}
    Since $A_\natural/\fm A_\natural\neq 0$,
    it suffices to show $B\to A_\natural$ is flat.

    We know $A\to A_\natural$ is flat \cite[Corollary 3.3.3]{SchoutensBOOK}.
    By \citestacks{051C}, we may base change to the reduction of $A$ and assume $A$ reduced.
    If $\dim A=0$, then $A=k$, so $B=l$ is a field. 
    Thus we may assume $\dim A>0,$
    so $\depth A>0.$

    We need to show $N\otimes_B^L A_\natural$ is concentrated in degree $0$
    for all finite $B$-modules $N$.
    If $N=l$,
    then since $\fn=\fm B$,
    $N\otimes_B^L A_\natural=k\otimes_A^L A_\natural$
    is concentrated in degree $0$ as $A\to A_\natural$ is flat.
    Thus $N\otimes_B^L A_\natural$ is concentrated in degree $0$
    for all $N$ of finite length.

    For a general $N$, we may assume, by Noetherian induction,
    that $\overline{N}\otimes_B^L A_\natural$ is concentrated in degree $0$
    for all proper quotients $\overline{N}$ of $N$.
    Since $N[\fn^\infty]$ is of finite length,
    we may thus assume $N[\fn^\infty]=0$,
    \emph{i.e.}, $\depth N>0$.
    
    Since $\depth A>0$ and since $\fm B=\fn$,
    there exists an $f\in \fm$ that is both a nonzerodivisor in $A$ 
    (thus a noninvertible nonzerodivisor in both $B$ and $A_\natural$)
    and a nonzerodivisor in $N$.
    Thus $N\otimes_B^L A_\natural/fA_\natural\cong N/fN \otimes_B^L A_\natural$
    is concentrated in degree $0$.
    Since $N$ and thus 
    $N\otimes_B^L A_\natural$
    can be represented by a complex of finite free modules,
    Lemma \ref{lem:modabddImpliesbdd} implies that $N\otimes_B^L A_\natural$ is concentrated in degree $0$.
\end{proof}

\begin{Lem}\label{lem:ExtVanishFurther}
    Let $(A,\fm,k)$ be a Noetherian local ring,
    $d=\dim A$.
    Let $b\in\bZ$ and let $M\in D^{\leq b}_{Coh}(A)$.
    If $\operatorname{Ext}^m_{A}(k,M)=0$ for all $b+1\leq m\leq b+d+1$,
    then $\operatorname{Ext}^m_{A}(N,M)=0$ for all $m>b$ and all finite $A$-modules $N$.
\end{Lem}
\begin{proof}
    We show by induction on $h=\dim\operatorname{Supp} N$ that $\operatorname{Ext}^m_{A}(N,M)=0$ for all $b+1\leq m\leq b+d+1-h$,
    for all finite $A$-modules $N$.
    This implies the result, since then $\operatorname{Ext}^{b+1}_{A}(N,M)=0$ for all finite $A$-modules $N$, so dimension shifting and the fact $M\in D^{\leq b}(A)$ shows $\operatorname{Ext}^{m}_{A}(N,M)=0$ for all finite $A$-modules $N$ and all $m>b$.
    
    If $h=0$, then $N$ has finite length, thus is a successive extension of $k$ and the vanishing holds by the assumption.
    Assume $h>0$, so there exists $f\in \fm$ such that $\dim\left( V(fA)\cap \operatorname{Supp} N\right)<h$.
    Let $C=\operatorname{Cone}(N\xrightarrow{\times f}N)$,
    so we have a distinguished triangle
    $N\xrightarrow{\times f}N\to C\to +1$.
    We note that $C\in D^{[-1,0]}_{Coh}(A)$ and that the dimension of the support of the cohomology modules of $C$ is less than $h$.
    Thus the induction hypothesis tells us 
    $\operatorname{Ext}^m_{A}(C,M)=0$ for all $b+2\leq m\leq b+d+2-h$.
    The distinguished triangle above and Nakayama's Lemma tells us $\operatorname{Ext}^m_{A}(N,M)=0$ for all $b+1\leq m\leq b+d+1-h$, as desired.
\end{proof}

\begin{Cor}\label{cor:CriterionForDual}
    Let $(A,\fm,k)$ be a Noetherian local ring of dimension $d$,
    and let $a\in \bZ$.
    Let $M\in D^{[a,0]}_{Coh}(A)$.
    Assume that 
    \begin{align*}
    \tau_{\leq d+1}R\operatorname{Hom}_{A}(k,M)&\cong k.\\
    \tau_{\geq a}\tau_{\leq -a}R\operatorname{Hom}_{A}(M,M)&=A.
    \end{align*}
    Then $M$ is a normalized dualizing complex.
\end{Cor}
\begin{proof}
    By Lemma \ref{lem:ExtVanishFurther},
    $\operatorname{Ext}^m_A(N,M)=0$ for all finite $A$-modules $N$ and all $m>0$.
    Thus $M$ has injective amplitude $[a,0]$ and $R\operatorname{Hom}_{A}(k,M)\cong k$.
    Since $M\in D^{[a,0]}(A)$,
    we see $R\operatorname{Hom}_{A}(M,M)\in D^{[a,-a]}(A)$.
    Thus $R\operatorname{Hom}_{A}(M,M)=A$,
    so $M$ is a normalized dualizing complex.
\end{proof}

\begin{Thm}\label{thm:HenselHasDual}
    Let $(A,I)$ be a Henselian pair where $(A,\fm,k)$ is a Noetherian local ring.
    Assume that
    \begin{enumerate}[label=$(\roman*)$]
        \item\label{hensdual_reg} The $I$-adic completion map $A\to B:=\lim_n A/I^n$ is regular.
        \item $B$ admits a dualizing complex. 
    \end{enumerate}
    Then $A$ admits a dualizing complex.
\end{Thm}
Note that \ref{hensdual_reg} is always true for quasi-excellent $A$, see \citestacks{0AH2}.
\begin{proof}
  %
%

    By Popescu's theorem \citestacks{07GC}, $B$ is the colimit of a direct system of smooth $A$-algebras $A_i$.
    The composition $A_i\to B\to B/IB\cong A/I$ gives an $A$-algebra map $A_i\to A/I$,
    thus $A_i$ admits a section $A_i\to A$ since $(A,I)$ is Henselian, see \citestacks{07M7}.
    Thus \cite[Theorem 7.1.1]{SchoutensBOOK} applies (see also \cite[Lemma 3.4]{Lyu-Splinter}),
    so there exists an $A$-algebra map $B\to A_\natural$ where $A_\natural=A^X/\cU=S_\cU^{-1}A^X$ is an ultrapower of $A$.
    Here $X$ is a set, $\cU$ is an ultrafilter on $X$, 
    and $S_\cU=\{e_U\mid U\in \cU\}$,
    where $e_U$ is defined by $(e_U)_x=1$ when $x\in U$ and $(e_U)_x=0$ when $x\not\in U$.
    This map is flat by Theorem \ref{thm:MaptoUltraFlat}.

    Let $K\in D^b_{Coh}(B)$ be a normalized dualizing complex.
    Then $F:=K\otimes_B^L A_\natural\in D(A_\natural)$ is pseudo-coherent and bounded.
    Writing $k_\natural=k\otimes_A^L A_\natural=k\otimes_B^L A_\natural$,
    we have
    \begin{align*}
    R\operatorname{Hom}_{A_\natural}(k_\natural,F)&=R\operatorname{Hom}_{B}(k,K)\otimes_B^L A_\natural\cong k_\natural\\
    R\operatorname{Hom}_{A_\natural}(F,F)&=R\operatorname{Hom}_{B}(K,K)\otimes_B^L A_\natural= A_\natural
    \end{align*}
    by \citestacks{0A6A} and by the fact $K$ is a normalized dualizing complex.

    Let $d=\dim A$. 
    Then $K\in D^{[-d, 0]}(B)$ and thus $F\in D^{[-d, 0]}(A_\natural)$.
    By Lemma \ref{lem:DescendDbCoh},
    there exists an $e\in S_\cU$ and an $M\in D((A^X)_e)$ represented by a complex $M^\bullet$ of finitely presented $(A^X)_e$-modules
    such that $M\otimes^L_{(A^X)_e}A_\natural\cong F$,
    that $M^m=0$ for all $m>0$ and $m<-d$,
    and that $M^m$ is free for all $-d<m\leq 0$.
    Note that if $e=e_U$, then $(A^X)_e=A^U$,
    thus after replacing $X$ by a subset in $\cU$ we may assume $e=1$.

    We note that another application of Lemma \ref{lem:DescendDbCoh} with $a=-2d$ shows that, after replacing $X$ by a subset,
    $M$ is also represented by a complex of free $A$-modules $P^\bullet$ such that $P^m=0$ for all $m>0$ and that $P^m$ is finite for all $m\geq -2d-1$.
    Thus $R\operatorname{Hom}_{A^X}(M,M)$ is represented by the Hom complex
    $E^\bullet:=\operatorname{Hom}^\bullet_{A^X}(P^\bullet,M^\bullet)$,
    and $E^m$ is finitely presented for all $m\leq d+1$. 
Also note that $R\operatorname{Hom}_{A^X}(k^X,M)$ is represented by a complex of finitely presented modules in degrees $\geq -d$,
since $k^X=A^X\otimes^L_A k$ is represented by a complex of finite free modules in degrees $\leq 0$.

    Let $T^{-1}\to T^0\to T$ be a three-term complex of finitely presented
    $A^X$-modules. 
    Fix presentations $P^{i,-1}\to P^{i,0}\to T^{i}\to 0\ (i=-1,0,1)$, where $P^{i,j}$ are finite free, and maps $P^{-1,0}\to P^{0,0}\to P^{1,0}$ lifting $T^{-1}\to T^0\to T^1$.
    We do not require the composition $P^{-1,0}\to P^{1,0}$ to be zero. 
    Fix bases of $P^{i,j}$,
    so the maps between $P^{i,j}$  are represented by matrices.
    The statement that the cohomology $\ker(T^0\to T^1)/\operatorname{im}(T^{-1}\to T^0)$ is $0$ (resp. $A^X,k^X$)
    is a first-order statement of the coefficients of the matrices.
    This observation and the observation on Hom complexes above tells us that after replacing $X$ by a subset in $\cU$,
    we have
    \begin{align*}
    \tau_{\leq d+1}R\operatorname{Hom}_{A^X}(k^X,M)&\cong k^X,\\
    \tau_{\geq -d}\tau_{\leq d}R\operatorname{Hom}_{A^X}(M,M)&=A^X.
    \end{align*}
    Apply the projection $A^X\to A$ to some coordinate, we get an object $M_1\in D^{[-d,0]}_{Coh}(A)$
    such that
    \begin{align*}
    \tau_{\leq d+1}R\operatorname{Hom}_{A}(k,M_1)&\cong k.\\
    \tau_{\geq -d}\tau_{\leq d}R\operatorname{Hom}_{A}(M_1,M_1)&=A.
    \end{align*}
    Then $M_1$ is a dualizing complex for $A$ by Corollary \ref{cor:CriterionForDual}.
\end{proof}

An \emph{elementary \'etale-local ring map} is a local map $(R,\fm,k)\to (S,\fn,l)$ of local rings such that
$S$ is the localization of an \'etale $R$-algebra at a prime ideal and that $l=k$.

\begin{Cor}\label{cor:DualAfterEtloc}
    Let $A$ be a Noethrian local G-ring.
    Then there exists an elementary \'etale-local ring map $A\to A'$ such that $A'$ admits a dualizing complex.
\end{Cor}
\begin{proof}
    Let $A^h$ be the Henselization of $A$, so $A^h$ is a Henselian G-ring \citestacks{07QR}.
    Thus $A^h$ admits a dualizing complex $K\in D(A^h)$ by Theorem \ref{thm:HenselHasDual} (and \citestacks{0BFR} or our Theorem \ref{thm:GLOBALdeform0}).
%
    Since $A\to A^h$ is the filtered colimit of elementary \'etale-local ring maps $A\to A'$
    \citestacks{04GN}
    and since each $A'\to A^h$ is flat
    (cf. \citestacks{08HS}),
    by Lemma \ref{lem:DescendDbCoh}
    there exists an $A'$ 
    and a $K'\in D^b_{Coh}(A')$ 
    such that $K'\otimes^L_{A'}A^h=K$.
    Since $K$ is a dualizing complex for $A^h$, 
    by flatness again we see $K'$ is a dualizing complex for $A'$, see \citestacks{0E4A}.
\end{proof}

\section{Pseudo-dualizing complexes}\label{sec:PesudoDual}

\begin{Def}\label{def:weakDual}
    Let $A$ be a Noetherian ring.
    Let $K\in D(A)$.
    We say $K$ is a \emph{pseudo-dualizing complex}
    if $K\in D^b_{Coh}(A)$ and $K_\fp$ is a dualizing complex for $A_\fp$ for all $\fp\in\Spec(A)$.
    Thus $A$ is Gorenstein if and only if $A\in D(A)$ is a pseudo-dualizing complex.

    Let $X$ be a Noetherian scheme.
    $K\in D(X)$ is a \emph{pseudo-dualizing complex}
    if $K\in D^b_{Coh}(X)$ and $K_x$ is a dualizing complex for $\cO_{X,x}$ for all $x\in X$.
\end{Def}

It is clear that $K\in D^b_{Coh}(A)$ is a pseudo-dualizing complex if and only if $K_\fm$ is a dualizing complex for $A_\fm$ for all maximal ideals $\fm$ of $A$.

\begin{Rem}\label{rem:DualRemAfterDef}
    A pseudo-dualizing complex is a dualizing complex if and only if $A$ (resp. $X$) is finite-dimensional.
    See \cite[Chapter V, Proposition 8.2]{Residues-Duality}.

    The existence of a pseudo-dualizing complex implies being catenary, in fact implies the existence of a codimension function, see \cite[p. 287]{Residues-Duality}.
\end{Rem}

\begin{Lem}\label{lem:PseuDualBdd}
    Let $A$ be a Noetherian ring.
    Assume that 
    for every $\fp\in\Spec(A)$,
    $\Spec(A/\fp)$ contains a nonempty Cohen-Macaulay open subscheme.
    Let $K\in D(A)$ be such that $K_\fp$ is a dualizing complex for all $\fp\in \Spec(A)$.
    Then $K$ is bounded. 
\end{Lem}
\begin{proof}
    For all $\fp\in\Spec(A)$, we have
    \[
    \HT\fp-\depth A_\fp=\max \{b-a\mid a\leq b, H^a(K_\fp)\neq 0, H^b(K_\fp)\neq 0\}.   
    \]
By \cite[Proposition 6.10.6]{EGA4_2}, the function on the left hand side
    is constructible on $\Spec(A)$.
    Let $N$ be the maximum of this function,
    and for each minimal prime $\fq$ of $A$ let $c_\fq$ be the unique integer such that $H^{c_\fq}(K_\fq)\neq 0$.
    Then it is clear that $K$ is concentrated in degrees 
    $[-N+\min_\fq c_\fq,\max_\fq c_\fq+N]$.
    (In fact, $[\min_\fq c_\fq,\max_\fq c_\fq+N]$.) 
\end{proof}

\begin{Lem}\label{lem:FiniteUpperShriek}
    Let $f:A\to B$ be a finite map of Noetherian rings.
   Let $K\in D(A)$ be a pseudo-dualizing complex.
    Then $L:=R\operatorname{Hom}_A(B,K)\in D(B)$ is a pseudo-dualizing complex.
\end{Lem}
\begin{proof}
    Since $K$ is a pseudo-dualizing complex,
    we have, for all $\fp\in\Spec(A)$,
    \[
    \HT\fp-\depth A_\fp=\max \{b-a\mid a\leq b, H^a(K_\fp)\neq 0, H^b(K_\fp)\neq 0\}.   
    \]
    Since $K\in D^b_{Coh}(A)$, the function on the right hand side is constructible.
    Therefore \cite[Proposition 6.10.6]{EGA4_2} shows that for every $\fp\in\Spec(A)$,
    $\Spec(A/\fp)$ contains a nonempty Cohen-Macaulay open subscheme.
    Since $B$ is finite over $A$ the same is true for $B$.
    Since $K$ is bounded, $L\in D_{Coh}(B)$.
    It is clear that $L_\fq$ is a dualizing complex for $B_\fq$ for all $\fq\in\Spec(B)$.
    We conclude by Lemma \ref{lem:PseuDualBdd}.
\end{proof}

Lemma \ref{lem:FiniteUpperShriek} has a number of consequences that tell us pseudo-dualizing complexes resemble various properties of dualizing complexes.

\begin{Cor}\label{cor:UpperShriek}
    Let $f:X\to Y$ be a morphism separated of finite type between Noetherian schemes.
    Then $f^!$ sends a pseudo-dualizing complex for $Y$ to a pseudo-dualizing complex for $X$.
\end{Cor}
\begin{proof}
    Recall that $f^!$ is well-defined for all separated morphisms of finite type \citestacks{0AA0},
    compatible with composition \citestacks{0ATX},
    flat base change \citestacks{0E9U},
    and open immersion \citestacks{0AU0}.

    Let $K$ be a pseudo-dualizing complex for $Y$.
    Since the result is true
    after base change to $\Spec(\cO_{Y,y})$ for all $y\in Y$ \citestacks{0AA3},
    we see $(f^!K)_x$ is a dualizing complex for $\cO_{X,x}$ for all $x\in X$.
    Thus it suffices to show $f^!K\in D^b_{Coh}(X)$. 
    We may assume $X$ and $Y$ affine by compatibility with open immersion.
    By compatibility with composition,
    it suffices to treat the cases $X=\bA^1_Y$ and $f$ is finite (or even closed immersion).
    If $X=\bA^1_Y$ the result is trivial, see \citestacks{0AA1}.
    If $f$ is finite, the result is true by \citestacks{0AA2} and Lemma \ref{lem:FiniteUpperShriek}.
    Therefore the result is true for any $f$.
\end{proof}

\begin{Def}\label{def:Gor-2}
    Let $A$ be a Noetherian ring.
    We say $A$ is \emph{Gor-2} if $\Spec(A/\fp)$ has a nonempty Gorenstein open subset for all $\fp\in\Spec(A)$.
    This is equivalent to that the Gorenstein locus of any finite type $A$-algebra is open,
    see \cite[Proposition 1.7]{Gor-2-rings}.
    We recover this fact in Corollaries \ref{cor:Gor2preserved} and \ref{cor:Gor2impliesopen},
    and we will not use it.

    A locally Noetherian scheme is \emph{Gor-2} if it can be covered by affine opens $\Spec(A)$ where $A$ is Gor-2.
\end{Def}

\begin{Cor}\label{cor:DualFGAlgGor2}
    Let $A$ be a Noetherian ring that admits a pseudo-dualizing complex.
    Then every finite type $A$-algebra admits a pseudo-dualizing complex,
    and $A$ is Gor-2 and universally catenary.
    In particular, a Gorenstein ring is Gor-2.
\end{Cor}
\begin{proof}
    The ``in particular'' statement follows from the fact that $A$ is a pseudo-dualizing complex for a Gorenstein $A$.

    Every finite type $A$-algebra admits a pseudo-dualizing complex by Corollary \ref{cor:UpperShriek},
    so $A$ is universally catenary by Remark \ref{rem:DualRemAfterDef}.
    To see $A$ is Gor-2, we may assume by Lemma \ref{lem:FiniteUpperShriek} that $A$ is an integral domain, and we must show $A_f$ is Gorenstein for some nonzero $f$.
    This is clear: a pseudo-dualizing complex $K\in D^b_{Coh}(A)$ is a shift of the fraction field at the generic point of $A$,
    thus for some $f\neq 0$, $K_f$ is a shift of $A_f$ by Lemma \ref{lem:0A6Aforiso}.
\end{proof}

\begin{Cor}\label{cor:Gor2preserved}
    Let $A$ be a Noetherian ring that is Gor-2.
    Then every finite type $A$-algebra is Gor-2.
\end{Cor}
\begin{proof}
    By definition, we may assume $A$ Gorenstein, and we conclude by Corollary \ref{cor:DualFGAlgGor2}.
\end{proof}

\begin{Cor}\label{cor:DualizingFunctorGood}
    Let $A$ be a Noetherian ring, $K\in D(A)$ a pseudo-dualizing complex.
    Then the functor $D_K(-)=R\operatorname{Hom}_A(-,K)$
    maps $D^b_{Coh}(A)$ into $D^b_{Coh}(A)$,
    and the canonical map $\operatorname{id}\to D_K\circ D_K$ is an isomorphism of functors.
\end{Cor}
\begin{proof}
    By Lemma \ref{lem:FiniteUpperShriek},
    $D_K(A/\fp)\in D^b_{Coh}(A)$ for all $\fp\in\Spec(A)$.
    Thus $D_K$
    maps $D^b_{Coh}(A)$ into $D^b_{Coh}(A)$.
    The fact that $\operatorname{id}\to D_K\circ D_K$ is an isomorphism can be checked locally by \citestacks{0A6A},
    and the local case is well-known, cf. \citestacks{0A7C}.
\end{proof}

\begin{Rem}
    The functor $D_K(-)$ does not map $D^+_{Coh}(A)$ into either $D^-(A)$ or $D_{Coh}(A)$ when $A$ is not finite-dimensional.
    To see this, for $\fp\in\Spec(A)$ let $a_\fp$ be the smallest integer with $H^{a_\fp}(K_\fp)\neq 0$.
    Then $\fp\mapsto a_\fp$ is a bounded function since $K$ is bounded,
    and $R\operatorname{Hom}_{A}(A/\fm,K)\cong A/\fm[-a_\fm-\HT\fm]$
    for all maximal ideals $\fm$ of $A$.
    Thus if we pick $\fm_n$ so that $\HT\fm_n\geq 2n$ then 
    \[D_K\left(\bigoplus_n A/\fm_n[-n]\right)\cong \bigoplus_n A/\fm_n[-a_{\fm_n}+n-\HT\fm_n]\] is not in $D^-(A)$;
    and if we pick $\fm_n$ so that $\HT\fm_n> \HT\fm_{n-1}$
    then \[D_K\left(\bigoplus_n A/\fm_n[-\HT\fm_n]\right)\cong \bigoplus_n A/\fm_n[-a_{\fm_n}]\] is not in $D_{Coh}(A)$.

    On the other hand, $D_K(-)$ always map $D^-(A)$ (resp. $D^-_{Coh}(A)$) into $D^+(A)$ (resp. $D^+_{Coh}(A)$), as this is true for any $K\in D^+(A)$ (resp. $D^+_{Coh}(A)$).
\end{Rem}

\begin{Cor}\label{cor:DualizeUniqueuptoL}
    Let $A$ be a Noetherian ring, $K, K'\in D(A)$ two pseudo-dualizing complexes.
    Then there exists an invertible object $L\in D(A)$
    such that $K'\cong K\otimes^L_A L.$
\end{Cor}
\begin{proof}
    Let $L=R\operatorname{Hom}_A(K,K')$.
    Then $L=D_{K'}\circ D_{K}(A)$ is an invertible object by Corollary \ref{cor:DualizingFunctorGood} and \citestacks{0A7E}.
    The statement that the canonical map $K\otimes^L_A L\to K'$ is an isomorphism can be checked locally by \citestacks{0A6A},
    and when $A$ is local it follows from \citestacks{0A69}.
\end{proof}

We can characterize the existence of a pseudo-dualizing complex as follows.
For finite-dimensional rings this is due to Kawasaki \cite{Kawasaki-arithmetic-Macaulay}.

\begin{Thm}\label{thm:SharpKawasaki}
    Let $A$ be a Noetherian ring.
    Then $A$ admits a pseudo-dualizing complex if and only if there exists a finite type $A$-algebra $B$ that is Gorenstein and admits a section $B\to A$.
\end{Thm}
\begin{proof}
``If'' follows from Lemma \ref{lem:FiniteUpperShriek}.

We proceed to show ``only if.''
    Every finite type $A$-algebra admits a pseudo-dualizing complex,
    Corollary \ref{cor:UpperShriek},
    so we may replace $A$ by any finite type $A$-algebra that admits a section.

    By Remark \ref{rem:DualRemAfterDef},
    Corollary \ref{cor:DualFGAlgGor2},
    and \citestacks{0AWY},
    $A$ is universally catenary, has a codimension function,
    is Gor-2, and has Gorenstein formal fibers.
    By \cite[p. 2738, proof of Theorem 1.3]{Kawasaki-quotient-of-Macaulay},
    there exists a finite type $A$-algebra $B$ that is Cohen-Macaulay and admits a section.   
    Replace $A$ by $B$, we may assume our $A$ is Cohen-Macaulay.

    Assume $A$ is Cohen-Macaulay.
    We may assume $\Spec(A)$ is connected,
    and $K$ is concentrated in degree $0$.
    Then the square-zero extension $A\oplus H^0(K)$ is Gorenstein by \cite{Reiten-AplusK} (see also \cite[Corollary 2.12]{Aoyama-CanMod}), finishing the proof.
\end{proof}

\begin{Rem}
    Since our $A$ actually has a pseudo-dualizing complex,
    we can avoid the materials in \cite{Kawasaki-quotient-of-Macaulay};
    the materials in \cite{Kawasaki-arithmetic-Macaulay}
    are sufficient with a minor twist.
    The constructions are the same;
    \cite{Kawasaki-quotient-of-Macaulay} proved it works in a greater generality.
\end{Rem}

\section{Formal lifting of pseudo-dualizing complexes}\label{sec:Lift}
In this section we prove the following theorem,
which, in particular, recovers the existence of dualizing complexes for complete local rings.
\begin{Thm}\label{thm:GLOBALdeform}
    Let $A$ be a Noetherian ring, $I$ an ideal of $A$.
    Assume that $A$ is $I$-adically complete. 
    If $A/I$ admits a pseudo-dualizing (resp. dualizing) complex $K_1$,
    then $A$ admits a pseudo-dualizing (resp. dualizing) complex $K$ such that $R\operatorname{Hom}_{A}(A/I,K)\cong K_1$.
\end{Thm}

Note that this implies $A$ is universally catenary by Corollary \ref{cor:DualFGAlgGor2}.
Compare with \cite[Proposition 1.1]{Greco-Universal-Catenary-No-Lift}.

From Lemma \ref{lem:FiniteUpperShriek} and Theorem \ref{thm:HenselHasDual}
we get the following immediate
\begin{Cor}
    Let $(A,I)$ be a Henselian pair where $(A,\fm,k)$ is a Noetherian local ring.
    Assume that the $I$-adic completion map $A\to \lim_n A/I^n$ is regular.
    Then $A$ admits a dualizing complex if and only if $A/I$ does.
\end{Cor}

We need some preparations for Theorem \ref{thm:GLOBALdeform}.

\begin{Lem}\label{lem:radReflect}
    Let $A$ be a Noetherian ring, $I$ an ideal of $A$.
    Assume that $I$ is contained in the Jacobson radical of $A$.
    
    Let $K\in D^b_{Coh}(A)$.
    If $R\operatorname{Hom}_{A}(A/I,K)$ is a pseudo-dualizing complex for $A/I$,
    then $K$ is a pseudo-dualizing complex for $A$.
\end{Lem}
\begin{proof}
    We may assume $(A,\fm,k)$ local.
    We have
    $$R\operatorname{Hom}_{A}(k,K)=R\operatorname{Hom}_{A/I}(k,R\operatorname{Hom}_{A}(A/I,K))$$
    is a shift of $k$,
    so $K$ is a dualizing complex by \cite[Chapter V, Proposition 3.4]{Residues-Duality}.
\end{proof}

\begin{Lem}\label{lem:f!Cohf*Coh}
    Let $A$ be a Noetherian ring, $f\in A$.
    Let $M\in D^{b}(A)$.
    Consider the following conditions.
    \begin{enumerate}[label=$(\roman*)$]
        \item\label{f!f*:f!} $R\operatorname{Hom}_A(A/fA,M)\in D_{Coh}(A/fA)$.
        \item\label{f!f*:cone} $R\operatorname{Hom}_A(\operatorname{Cone}(A\xrightarrow{\times f}A),M)\in D_{Coh}(A/fA)$.
        \item\label{f!f*:f*} $M\in D_{Coh}(A)$.
    \end{enumerate}
    Then \ref{f!f*:f!} implies \ref{f!f*:cone},
    and if $A$ is $f$-adically complete and $M$ is derived $f$-adically complete, then
    \ref{f!f*:cone} implies \ref{f!f*:f*}.
\end{Lem}
\begin{proof}
Note that $R\operatorname{Hom}_A(A/fA,M)\in D^+(A/fA)$ since $M\in D^+(A)$.
    If \ref{f!f*:f!} holds, then $R\operatorname{Hom}_A(A/fA,M)\in D^+_{Coh}(A/fA)$, so for any $X\in D^b_{Coh}(A/fA)$,
    \[
    R\operatorname{Hom}_A(X,M)=R\operatorname{Hom}_{A/fA}(X,R\operatorname{Hom}_A(A/fA,M))
    \] 
    is in $D_{Coh}(A/fA)$,
in particular this holds for $X=C:=\operatorname{Cone}(A\xrightarrow{\times f}A)$.
    Thus \ref{f!f*:cone} holds.

    Now assume \ref{f!f*:cone},
    and assume $A$ is $f$-adically complete and $M$ is derived $f$-adically complete.
    From the distinguished triangle 
    $A\xrightarrow{\times f} A\to C\to +1$
    We know $R\operatorname{Hom}_A(C,M)\cong \operatorname{Cone}(M\xrightarrow{\times f}M)[-1]$,
    so $\operatorname{Cone}(M\xrightarrow{\times f}M)\in D_{Coh}(A/fA)$.
    Let $b$ be an integer such that  $M\in D^{\leq b}(A)$.
    Then $H^b(M)/fH^b(M)=H^b(\operatorname{Cone}(M\xrightarrow{\times f}M))$ is finite.
    Thus there exists a finite free $A$-module $F$ and a map $F[-b]\to M$ in $D(A)$ inducing a surjective map on $H^b(-)/fH^b(-)$.
    Since $F[-b]$ is also derived $f$-adically complete \citestacks{091T},
    we see $F[-b]\to M$ induces a surjective map on $H^b(-)$ by \citestacks{09B9}.
    Thus $\operatorname{Cone}(F[-b]\to M)\in D^{\leq b-1}(A)$,
    and we see inductively $H^c(M)$ is finite for all $c$.
\end{proof}

The following two lemmas are not used in the case $A$, or equivalently $A/I$, is finite-dimensional.

\begin{Lem}[{\cite[Lemma 6.4.3.7 and Proposition 6.6.4.6]{SAG}}]\label{lem:nilReflect}
    Let $A$ be a Noetherian ring, $I$ a nilpotent ideal of $A$.
    
    Let $K\in D^+(A)$, $L=R\operatorname{Hom}_{A}(A/I,K)$.
    Then the followings hold.
    \begin{enumerate}[label=$(\roman*)$]
        \item\label{nil:Coh} If $L\in D_{Coh}(A/I)$, then $K\in D_{Coh}(A)$.
        \item\label{nil:Dual} If $L$ is a pseudo-dualizing complex for $A/I$,
        then $K$ is a pseudo-dualizing complex for $A$.
    \end{enumerate}
\end{Lem}
\begin{proof}
    For \ref{nil:Coh}, we may assume $I=fA$ principal.
    Let $a$ be the smallest integer such that $H^a(K)\neq 0$.
    Then $H^a(K)[f]=H^a(L)$ is finite.
    The exact sequences
    $0\to H^a(K)[f]\to H^a(K)[f^{n+1}]\xrightarrow{\times f} H^a(K)[f^n]$
    tells us $H^a(K)$ is finite.
    Apply dimension shifting to the  distinguished triangle 
    \[
    H^a(K)[-a]\to K\to \tau_{>a}K\to +1
    \]
    we see inductively $H^c(K)$ is finite for all $c\in\bZ$.

    For \ref{nil:Dual}, from Corollary \ref{cor:DualFGAlgGor2} we know $A/I$ is Gor-2, so $A$ is Gor-2 by definition, since $I$ is nilpotent.
    By Lemma \ref{lem:PseuDualBdd}, it suffices to show \ref{nil:Dual} when $(A,\fm,k)$ is local.
    Then $$R\operatorname{Hom}_{A}(k,K)=R\operatorname{Hom}_{A/I}(k,R\operatorname{Hom}_{A}(A/I,K))$$
    is a shift of $k$,
    so $K$ has finite injective dimension \citestacks{0AVJ},
    in particular bounded.
    Since $K\in D_{Coh}(A)$ by \ref{nil:Coh}, \ref{nil:Dual} follows from Lemma \ref{lem:radReflect}.
\end{proof}

\begin{Lem}\label{lem:coprofEstimate}
    Let $A$ be a Noetherian ring,
    $\fp\in\Spec(A),f\in \fp$.
    Assume that $\Spec(A/\fp)$ contains a nonempty Cohen-Macaulay open subscheme.
    Then there exists $g\not\in\fp$ such that for all $\fq\in V(\fp)\cap D(g),$
    \[
    \HT(\fq/f^nA)-\depth \left(A_\fq/f^nA_\fq\right)\leq \HT(\fp)
    \]
    for all $n\in\bZ_{\geq 1}$.
\end{Lem}
\begin{proof}
    Let $N\in\bZ_{\geq 1}$ be such that $J:=A[f^\infty]=A[f^N]$.
    After replacing $A$ by some $A_g$, we may assume $A/\fp$ Cohen-Macaulay and
    our inequality holds for all $\fq\in V(\fp)$ and all $n< N$ by \cite[Proposition 6.10.6]{EGA4_2}.
    By the same proposition may also assume that, for all $\fq\in V(\fp)$,
        $\depth J_\fq\geq \HT(\fq/\fp)$ and 
        $\depth (A_\fq/J_\fq)\geq \HT(\fq/\fp)+1$.
    Since $f$ is a nonzerodivisor on $A/J$ we have $\depth (A_\fq/(f^nA+J)_\fq)\geq \HT(\fq/\fp)$ for all $n\in\bZ_{\geq 1}$.

    For $n\geq N$, the sequence
    \[\begin{CD}
        0@>>> J@>>> A/f^nA @>>> A/(f^nA+J)@>>> 0
    \end{CD}\]
    is exact.
    Thus the depth inequalities above imply
    $\depth (A_\fq/f^nA_\fq)\geq \HT(\fq/\fp)$ for all $n\in\bZ_{\geq N}$.
    Now for all $n\geq N$ we have
    $ \HT(\fq/f^nA)-\depth \left(A_\fq/f^nA_\fq\right)\leq \HT(\fq)-\HT(\fq/\fp)$.
    The right hand side equals $\HT(\fp)$ after localization by \cite[Proposition 6.10.6]{EGA4_2}.
\end{proof}

\begin{proof}[Proof of Theorem \ref{thm:GLOBALdeform}]
A pseudo-dualizing complex is a dualizing complex if and only if the ring is finite-dimensional, see Remark \ref{rem:DualRemAfterDef}.
Since $I$ is in the Jacobson radical of $A$,
$A$ is finite-dimensional if and only if $A/I$ is.
Thus it suffices to prove the result for pseudo-dualizing complexes.

We may assume $I=fA$ principal. Let $A_n=A/f^nA$.

    Let $(J_1^\bullet,d_1^\bullet)$ be a bounded below complex of injective $A_1$-modules that represents $K_1$.
    Fix $a\in\bZ$ so that $J_1^m=0$ for all $m<a$. 
    Further, fix $c_0\in\bZ_{\geq a}$ such that for all minimal primes $\fq$ of $fA$,
    there exists $c\leq c_0$ such that 
    $H^c(J_1^\bullet)_\fq\neq 0$.
    
    For each $m$, 
    let $J^m$ be an injective hull of $J_1^m$ as an $A$-module.
    Thus $J^m$ is an essential extension of $J_1^m$ and $J_1^m=J^m[f]$.
    Let $J^m_\infty=J^m[f^\infty]$.
    Then $J^m_\infty$ is an injective $A$-module by \citestacks{08XW}.

\begin{Clm}
\label{clm:extendd}
    There exist maps $d^m_\infty:J^m_\infty\to J^{m+1}_\infty$ extending $d^m_1$
    such that $(J^\bullet_\infty,d^\bullet_\infty)$ is a complex.
\end{Clm}

    The proof is given after the main argument.
    Granting Claim \ref{clm:extendd}, $J^\bullet_\infty$ is now a bounded below complex of $f^\infty$-torsion injective $A$-modules.
    We next show that $J^\bullet_\infty\in D(A)$ is bounded.
    Note that if $\dim A_1$ is finite,
    then $K_1$ has finite injective dimension,
    so we could choose $J_1^\bullet$ so that $J_1^m=0$ for $m\gg 1$,
    and $J_\infty$ is automatically bounded.
    Without this assumption,
    writing $J^m_n=J^m[f^n]$, 
    we have that $R\operatorname{Hom}_{A_n}(A_1,J^\bullet_n)=J^\bullet_n[f]=J^\bullet_1\in D(A_1)$, since $J^\bullet_n$ is a bounded below complex of injectives.
    Thus $J^\bullet_n$ is a pseudo-dualizing complex for $A_n$ by Lemma \ref{lem:nilReflect}.
    Let $\fq$ be a minimal prime of $fA$.
    Then $(J^\bullet_n)_\fq$ is exact except at a single degree $c$,
    and applying $R\operatorname{Hom}_{A_\fq/f^n A_\fq}(A_\fq/f A_\fq,-)$ we see $(J^\bullet_1)_\fq$ is exact except at  degree $c$,
    so $c\leq c_0$.
    Thus
    for all $\fp\in V(fA)$, $H^c(J^\bullet_n)_\fp\neq 0$ for some $c\leq c_0$.
    
    Note that
    $A/fA$ is Gor-2 by Corollary \ref{cor:DualFGAlgGor2},
    so Lemma \ref{lem:coprofEstimate} shows that there exists $b\in\bZ_{\geq 1}$ such that for all $\fp\in V(fA)$ and all $n\in\bZ_{\geq 1}$,
    \[
     \HT(\fp/f^nA)-\depth \left(A_\fp/f^nA_\fp\right)\leq b.
    \]
    Therefore $(J^\bullet_n)_\fp\in D^{[a,b+c_0]}(A)$,
    so $(J^\bullet_\infty)_\fp\in D^{[a,b+c_0]}(A)$ for all $\fp\in V(fA)$.
    Since $f$ is in the Jacobson radical of $A$ we have $J^\bullet_\infty\in D^{[a,b+c_0]}(A)$.
    
    Let $K\in D(A)$ be the derived $f$-adic completion of $J^\bullet_\infty$,
    so $K\in D^b(A)$ by for example \citestacks{091Z}.
    By \citestacks{0A6Y} (and \citetwostacks{091T}{0A6R}) we have
    \[
    R\operatorname{Hom}_A(A/fA,K)=R\operatorname{Hom}_A(A/fA,J^\bullet_\infty),
    \]
    and the right hand side is just $J^\bullet_\infty[f]=J^\bullet_1$ since $J^\bullet_\infty$ is a bounded below complex of injectives.
    Thus $K\in D^b_{Coh}(A)$ by Lemma \ref{lem:f!Cohf*Coh}.
    We conclude that $K$ is a pseudo-dualizing complex for $A$ by Lemma \ref{lem:radReflect}.
\end{proof}
\begin{proof}[Proof of Claim \ref{clm:extendd}]
    We first show that there exist maps $d^m_2:J^m_2\to J^{m+1}_2$ extending $d^m_1$
    such that $(J^\bullet_2,d^\bullet_2)$ is a complex. 
    This follows from the dual version of \citestacks{0DYR}.
    We give the proof in our case for the reader's convenience.

    Let $\delta^m_2:J^m_2\to J^{m+1}_2$ 
    be arbitrary maps extending $d^m_1$.
    Composing, we get maps $\delta^{m+1}_2\circ\delta^m_2:J^m_2\to J^{m+2}_2$.
    Since $(J^\bullet_1,d^\bullet_1)$ is a complex,
    $\delta^{m+1}_2\circ\delta^m_2$ is zero on $J^m_1=J^m_2[f]$.
    Since $f^2=0\in A_2$, we have $fJ^m_2\subseteq J^m_1$,
    so $\operatorname{im}(\delta^{m+1}_2\circ\delta^m_2)\subseteq J^{m+2}_1.$
    This tells us $(J^\bullet_2/J^\bullet_1,\delta^\bullet_2)$ is a complex,
    and that $\delta^{m+1}_2\circ\delta^m_2$
    induces a map $J^m_2/J^m_1\to J^{m+2}_1$.
    It is clear that
    $\delta^{\bullet+1}_2\circ\delta^\bullet_2:J^\bullet_2/J^\bullet_1\to J^{\bullet+2}_1$
    is a map of complexes.

    Since $J^m_2$ is injective,
    we have canonical isomorphisms \[J^m_2/J^m_1=\operatorname{Hom}_{A_2}(fA_2,J^m_2)=\operatorname{Hom}_{A_1}(fA_2,J^m_1).\]
    Therefore $J^\bullet_2/J^\bullet_1$ represents $R\operatorname{Hom}_{A_1}(fA_2,J^\bullet_1)$,
    so $R\operatorname{Hom}_{A_1}(J^\bullet_2/J^\bullet_1,J^{\bullet+2}_1)=R\operatorname{Hom}_{A_1}(R\operatorname{Hom}_{A_1}(fA_2,J^\bullet_1),J^\bullet_1[2])=fA_2[2]$
    by Corollary \ref{cor:DualizingFunctorGood}.
    Thus $\delta^{\bullet+1}_2\circ\delta^\bullet_2:J^\bullet_2/J^\bullet_1\to J^{\bullet+2}_1$
    is zero in $D(A_1),$
    hence homotopic to zero (see \citestacks{05TG}).
    Let $g^\bullet_2:J^\bullet_2/J^\bullet_1\to J^{\bullet+1}_1$
    be a homotopy between $\delta^{\bullet+1}_2\circ\delta^\bullet_2$ and $0$.
    View each $g^m_2$ as a map $g^m_2:J^m_2\to J^{m+1}_2$,
    we see $d^m_2=\delta^m_2-g^m_2$
    is what we want.

    Now $J^\bullet_2$ is a bounded below complex of injective $A_2$-modules,
    so \[R\operatorname{Hom}_{A_2}(A_1,J^\bullet_2)=J^\bullet_2[f]=J^\bullet_1,\]
    thus $J^\bullet_2$ is a pseudo-dualizing complex for $A_2$ by Lemma \ref{lem:nilReflect}.
    The same argument as above tells us we can extend $d^\bullet_2$ to $d^\bullet_3$, ad infinitum, showing Claim \ref{clm:extendd}.
\end{proof}

\begin{Rem}
    We record the dual version of \citestacks{0DYR} in our mind for the reader's convenience:
    for a ring $A$ and an ideal $I$ with $I^2=0$,
    the obstruction for a $K\in D^+(A/I)$ to be of the form $R\operatorname{Hom}_A(A/I,K')$
    is a map $R\operatorname{Hom}_{A/I}(I,K)\to K[2]$ in $D(A/I)$.
    The author does not know a reference for this.
    
    For pseudo-dualizing complexes, this obstruction vanishes automatically by Corollary \ref{cor:DualizingFunctorGood}, as seen in the proof above.
    Alternatively, one can use the argument as in \cite[Lemma 6.6.4.9]{SAG},
    if willing to use animated rings.
\end{Rem}

\section{Openness of loci}\label{sec:Open}

For basics about canonical modules refer to \cite[\S 1]{Aoyama-CanMod}.

\begin{Lem}\label{lem:DualOpenInVp}
    Let $A$ be a Noetherian ring, $\fp\in\Spec(A)$.
    Assume $\Spec(A/\fp)$ contains a nonempty Gorenstein open subset.
    Then the followings hold.
    \begin{enumerate}[label=$(\roman*)$]

        \item\label{dualopenInVp:spreads} Let $K\in D^b_{Coh}(A)$.
        If $K_\fp$ is a dualizing complex for $A_\fp$,
         then there exists an $f\in A\setminus\fp$ such that $K_\fq$ is a dualizing complex for $A_\fq$ for all $\fq\in V(\fp)\cap D(f)$.


        \item\label{dualopenInVp:H0spreads} Let $M$ be a finite $A$-module.
        If $M_\fp$ is a canonical module for $A_\fp$,
         then there exists an $f\in A\setminus\fp$ such that $M_\fq$ is a canonical module for $A_\fq$ for all $\fq\in V(\fp)\cap D(f)$.
    \end{enumerate}
\end{Lem}
\begin{proof}
    We may assume $A/\fp$ Gorenstein.
    Let $B$ be the $\fp$-adic completion of $A$, so $A\to B$ is a flat ring map with $A/\fp=B/\fp B$,
    thus open subsets of $V(\fp)$ in $\Spec(A)$ are in one-to-one correspondence with open subsets of $V(\fp B)$ in $\Spec(B)$.
    Note that for all $\fP\in V(\fp)$,
    $A_\fP^\wedge=B_{\fP B}^\wedge$ since $A/\fp^n=B/\fp^n B$ for all $n$.
    Thus the base change of a dualizing complex for (resp. canonical module of)  $A_\fp$ to $B_{\fp B}$
    is  a dualizing complex for (resp. canonical module of) $B_{\fp B}$,
    and we may apply flat descent (\citestacks{0E4A} and \cite[Theorem 4.2]{Aoyama-CanMod})
    for $\fP\in V(\fp)$.
    Thus
    it suffices to prove the lemma in the case $A$ is $\fp$-adically complete and $A/\fp$ is Gorenstein.
    
    In this case, $A$ admits a pseudo-dualizing complex $E$ by Theorem \ref{thm:GLOBALdeform},
    and we have $K_\fp\cong E_\fp$ (resp. $M_\fp\cong H^0(E_\fp)$ and $E_\fp\in D^{\geq 0}(A_\fp)$; \cite[(1.5)]{Aoyama-CanMod})
    after a shift.
    After localizing we have $K\cong E$ by Lemma \ref{lem:0A6Aforiso} (resp. $M\cong H^0(E)$ and $E\in D^{\geq 0}(A)$),
    as desired. 
    \end{proof}

We say a finite module $M$ over a Noetherian ring $A$ a \emph{canonical module} if $M_\fp$ is a canonical module of $A$ for all $\fp\in \Spec(A)$.
Localization of a canonical module is a canonical module, see \cite[Corollary 4.3]{Aoyama-CanMod},
so it suffices to check at the maximal ideals of $A$.

\begin{Thm}\label{thm:DualOpen}
    Let $A$ be a Noetherian ring that is Gor-2.
    Then the followings hold.
    \begin{enumerate}[label=$(\roman*)$]
        \item\label{dualopen:exists} If $\fp\in \Spec(A)$ is such that $A_\fp$ admits a dualizing complex,
        then there exists an $f\in A\setminus\fp$ such that $A_f$ admits a pseudo-dualizing complex.

        \item\label{dualopen:spreads} Let $K\in D^b_{Coh}(A)$.
        If $\fp\in \Spec(A)$ is such that $K_\fp$ is a dualizing complex for $A_\fp$,
         then there exists an $f\in A\setminus\fp$ such that $K_f$ is a pseudo-dualizing complex for $A_f$.

          \item\label{dualopen:H0exists} If $\fp\in \Spec(A)$ is such that $A_\fp$ admits a canonical module,
        then there exists an $f\in A\setminus\fp$ such that $A_f$ admits a canonical module.

        \item\label{dualopen:H0spreads} Let $M$ be a finite $A$-module.
        If $\fp\in \Spec(A)$ is such that $M_\fp$ is a canonical module of $A_\fp$,
         then there exists an $f\in A\setminus\fp$ such that $M_f$ is a canonical module of $A_f$.
    \end{enumerate}
\end{Thm}
\begin{proof}
    We know \ref{dualopen:spreads} implies \ref{dualopen:exists}  and  \ref{dualopen:H0spreads} implies \ref{dualopen:H0exists} by Lemma \ref{lem:DescendDbCoh}.
    On the other hand, \ref{dualopen:spreads} and \ref{dualopen:H0spreads} follows from Lemma \ref{lem:DualOpenInVp} and general topology \citestacks{0541}.
\end{proof}

Apply to the case $K=A$ and $M=A$, we get the following.
The Gorenstein case is \cite[Proposition 1.7]{Gor-2-rings}, but the quasi-Gorenstein case seems to be new.

\begin{Cor}\label{cor:Gor2impliesopen}
    The Gorenstein and quasi-Gorenstein loci of a Gor-2 Noetherian ring is open.
\end{Cor}
\begin{Rem}
    If we assume in addition that $A$ is Cohen-Macaulay, 
    then Theorem \ref{thm:DualOpen} follows from
    the characterization of canonical modules \cite{Reiten-AplusK} (see also \cite[Corollary 2.12]{Aoyama-CanMod})
    and \cite[Proposition 1.7]{Gor-2-rings}.
\end{Rem}

\begin{Thm}\label{thm:etlocHasDual}
    Let $A$ be a Noetherian ring (resp. a Noetherian ring of finite dimension).
    Assume
    \begin{enumerate}
        \item $A$ is a G-ring.
        \item $A$ is Gor-2.
    \end{enumerate}
    Then there exists a faithfully flat, \'etale ring map $A\to A'$
    such that $A'$ admits a pseudo-dualizing complex (resp. a dualizing complex).
\end{Thm}
\begin{proof}
    Let $\fp\in\Spec(A)$.
    By Corollary \ref{cor:DualAfterEtloc},
    there exists an \'etale ring map $A\to B$ and a prime $\fq\in\Spec(B)$ lying above $\fp$ such that $B_\fq$ admits a dualizing complex. 
    Note that $B$ is Gor-2 since it is of finite type over $A$, Corollary \ref{cor:Gor2preserved}.
    By Theorem \ref{thm:DualOpen},
    localizing $B$ near $\fq$ 
    we may assume $B$ admits a pseudo-dualizing complex.
    If $A$ is finite-dimensional then so is $B$,
    so $B$ admits a dualizing complex.
    Now we take $A'$ to be a finite product of such $B$ so that $\Spec(A')\to \Spec(A)$ is surjective.
\end{proof}

\begin{Cor}\label{cor:QEetaleExcellent}
    Let $A$ be a Noetherian quasi-excellent ring.
    Then there exists a faithfully flat, \'etale ring map $A\to A'$
    such that $A'$ is excellent.
\end{Cor}

\begin{Rem}\label{rem:etaleUC}
    Corollary \ref{cor:QEetaleExcellent} is not difficult by itself, and may be well-known.
    We sketch an argument.
    First, consider a finite injective map $A\to B$ of Noetherian domains with $B$ universally catenary.
    Then for $\fp\in\Spec(A)$, $A_\fp$ is universally catenary if and only if for all $\fq\in\Spec(B)$ above $\fp$, $\HT(\fq)=\HT(\fp)$.
    This follows from \citetwostacks{02IJ}{0AW6}.
    This condition is constructible by for example \cite[Proposition 6.10.6]{EGA4_2}.
    
    Since a normal quasi-excellent ring is universally catenary \citestacks{0AW6},
    and since a quasi-excellent ring is Nagata \citestacks{07QV},
    we can take $B$ to be the normalization of $A=R/\fp$ where $\fp$ is a minimal prime of a given quasi-excellent ring $R$.
    It is now clear that the universally catenary locus of a quasi-excellent ring is open, 
    and that a unibranch quasi-excellent local ring is universally catenary.
    We thus get Corollary \ref{cor:QEetaleExcellent} from \citestacks{0CB4}.

    Using \cite[Lemma 2.5 and Theorem 2.13]{Macaulay-Cesnavi},
    replacing normalization with an $(S_2)$-ification,
    the argument above tells us that every  $(S_2)$-quasi-excellent \cite[(2.12)]{Macaulay-Cesnavi} Noetherian ring $A$
    has open universally catenary locus and 
    admits an \'etale, faithfully flat ring map $A\to A'$ such that $A'$ is $(S_2)$-excellent;
    and a Noetherian unibranch local ring with $(S_2)$ formal fibers is universally catenary.
\end{Rem}

\section{One-dimensional schemes}\label{sec:OneDim}

\begin{Lem}
\label{lem:DualAscend}    
    Let $(A,\fm,k)\to (B,\fn,l)$ be a flat local map of Noetherian local rings.
    Assume that $A$ admits a dualizing complex $K$, and that $B/\fm B$ is Gorenstein.
    Then $K\otimes_A^L B$ is a dualizing complex for $B$. 
\end{Lem}
\begin{proof}
    $R\operatorname{Hom}_A(k,K)$ is a shift of $k$, so $R\operatorname{Hom}_B(B/\fm B,K\otimes_A^L B)$ is a shift of $B/\fm B$ by \citestacks{0A6A}.
    We conclude by Lemma \ref{lem:radReflect}.
\end{proof}

\begin{Lem}\label{lem:OneDimLocHasDual}
    Let $A$ be a Noetherian local ring of dimension 1.
    Assume that the formal fibers of $A$ are Gorenstein.
    Then $A$ admits a dualizing complex.
\end{Lem}
\begin{proof}
%
    Let $a$ be a parameter of $A$.
    The $aA$-adic completion $A^\wedge$ of $A$ is a complete local ring,
    thus admits a dualizing complex $M$.
    We may assume that $M$ is normalized.
    
    Let $E$ be a dualizing complex for the Artinian ring $A_a$ concentrated in degree $-1$.
    The map $A_a\to (A^\wedge)_a$ is a flat map between Artinian rings with Gorenstein fibers by our assumptions,
    thus $E\otimes^L_{A_a}(A^\wedge)_a$ is a dualizing complex by Lemma \ref{lem:DualAscend}.
    On the other hand $M_a$
    is another dualizing complex of $(A^\wedge)_a$ concentrated in degree $-1$ by \citestacks{0A7V}.
    Thus
    $
    E\otimes^L_{A_a}(A^\wedge)_a\cong M_a
    $
     as $(A^\wedge)_a$ is Artinian.
    By \cite[Theorem 1.4]{Bha16-prod},
    there exists $K\in D(A)$ such that $K\otimes^L_A A^\wedge\cong M$.
    Then $K$ is a dualizing complex for $A$ \citestacks{0E4A}.
\end{proof}

\begin{Rem}
    Lemma \ref{lem:OneDimLocHasDual} also follows from \cite[Theorem 5.3]{one-dim-CM-has-dual} and, say, \cite[Corollary 3.7]{Ogoma-existcanmod-existomega}.
\end{Rem}

\begin{Thm}\label{thm:ExtendDualtoCod2} 
    Let $X$ be a Noetherian scheme. 
    Assume the followings.
    \begin{enumerate}
        \item $X$ is Gor-2.
        \item Every local ring of $X$ of dimension 1 has Gorenstein formal fibers.
    \end{enumerate}
    Let $U$ be an open subscheme of $X$, $K$ a pseudo-dualizing complex on $U$ such that for all generic points $\xi\in U$, $K_\xi$ is concentrated in degree $0$ (for example $U=\emptyset$).
    Then there exists an open subset $W\supseteq U$ of $X$ such that $\dim \cO_{X,x}>1$ for all $x\not\in W$ and that $W$ admits a pseudo-dualizing complex $K_W$ with $K_W|_{U}=K$ such that $K_\xi$ is concentrated in degree $0$ for all generic points $\xi\in X$.
\end{Thm}
\begin{proof}
%
    We may assume $(U,K)$ cannot be enlarged to any $(W,K_W)$, and we shall show $\dim \cO_{X,x}>1$ for all $x\not\in U$.

    Let $x\not\in U$, and assume $\dim\cO_{X,x}\leq 1$.
    Let $V=\Spec(A)$ be an affine open neighborhood of $x\in X$ such that $A$ has a pseudo-dualizing complex $L$, which exists by Lemma \ref{lem:OneDimLocHasDual} and Theorem \ref{thm:DualOpen}.
    We may assume that all generic points of $V$ specialize to $x$.
    Since $\dim \cO_{X,x}=1$,
    we see from \citestacks{0A7V} that, after shifting, 
    we may assume $L_\fq$ is concentrated at degree $0$ for all minimal primes $\fq$ of $A$.

    If $U\times_X \Spec(\cO_{X,x})$ is empty,
    then we may assume $U\cap V=\emptyset$, so we can enlarge $U$ to $U\cup V$.
    Otherwise $U\times_X \Spec(\cO_{X,x})$ has dimension zero, so it is affine.
    Thus we may assume $U\cap V$ is affine by \citestacks{01Z6}. 
    Write $B=\cO(U\cap V)$, so $B$ admits a dualizing complex $K_0$ given by the restriction of $K$ to $U\cap V$.
%
    Then $L\otimes^L_{A}B_\fp\cong K_0\otimes^L_B B_\fp$,
    where $\fp\in\Spec(A)$ corresponds to $x$,
    since $\dim B_\fp=0$
    and both sides are dualizing complexes concentrated in degree $0$.
    Thus there exists $g\in A\setminus \fp$ such that $L\otimes^L_{A} B_{g}\cong K_0\otimes^L_B B_{g}$ by Lemma \ref{lem:0A6Aforiso},
    so we can enlarge $U$ to $U\cup \Spec(A_g)$.
\end{proof}
\begin{Rem}
    The ring $A/I$ in \cite[\S 1]{Greco-Universal-Catenary-No-Lift} does not admit a dualizing complex
    by Theorem \ref{thm:GLOBALdeform}.
    $A/I$ is a semi-local ring with two maximal ideals $\fm',\fn'$
    with $\HT(\fm')=1$,
    $\HT(\fn')=2$.
    Both localizations $(A/I)_{\fm'}$
    and $(A/I)_{\fn'}$ admit dualizing complexes, see \cite[Lemma 1.5]{Greco-Universal-Catenary-No-Lift},
    but the restrictions of a dualizing complex for  $(A/I)_{\fn'}$
    to the generic points are concentrated in different degrees by \citestacks{0A7V},
    so Theorem \ref{thm:ExtendDualtoCod2} does not apply.
    Together with Theorem \ref{thm:DualOpen},
    this tells us that admitting dualizing complexes is not a Zariski local property.
\end{Rem}

\begin{Cor}\label{cor:Dim1HasDual}
    Let $X$ be a Noetherian scheme of dimension 1.
    Then $X$ admits a dualizing complex if and only if $X$ is Gor-2 and the local rings of $X$ has Gorenstein formal fibers.
\end{Cor}
\begin{proof}
    ``Only if'' follows from \citestacks{0AWY} and Corollary \ref{cor:DualFGAlgGor2}.
    ``If'' follows from Theorem \ref{thm:ExtendDualtoCod2} with $U=\emptyset$.
\end{proof}

\begin{Rem}
There exists a Noetherian local domain of dimension $1$ with non-Gorenstein formal fibers,
see \cite[Remarque 3.2]{Ferrand-Raynaud-bad-fibers}.
Such a ring is Gor-2, and even J-2, being one-dimensional and local.

There also exists a one-dimensional G-ring that is not Gor-2.
Such a ring can be constructed using the general method in \cite{Hoc73-nonopen}.

Therefore neither of the two conditions in Corollary \ref{cor:Dim1HasDual} implies the other.
\end{Rem}

\begin{Rem}
    In the terminology of \cite{Sharp-existomega-implies-acceptable},
    Corollary \ref{cor:Dim1HasDual} says that a Noetherian scheme of dimension 1 admits a dualizing complex if and only if it is acceptable.
\end{Rem}

\begin{Rem}
    There exists a two-dimensional excellent local UFD that does not admit a dualizing complex.
    See \cite[Example 6.1]{Nishimura-bad-local-rings}.
\end{Rem}

\section{Remarks on quotients of Gorenstein and Cohen-Macaulay rings}\label{sec:quot}

Our Lemma \ref{lem:FiniteUpperShriek} and Theorem \ref{thm:SharpKawasaki} imply the following result,
originally due to Kawasaki \cite{Kawasaki-arithmetic-Macaulay} for finite-dimensional rings.
\begin{Thm}\label{thm:Sharpiffs}
    Let $A$ be a Noetherian ring.
    Then the followings are equivalent.
    \begin{enumerate}[label=$(\roman*)$]
        \item $A$ is a quotient of a Gorenstein ring.
        \item There exists a finitely generated $A$-algebra $B$ that is Gorenstein and admits a section $B\to A$.
        \item $A$ admits a pseudo-dualizing complex.
    \end{enumerate}
\end{Thm}
Therefore, Theorem \ref{thm:GLOBALdeform} can be rewritten as 
\begin{Thm}\label{thm:QuotGorLifts}
    Let $A$ be a Noetherian ring, $I$ an ideal of $A$.
    If $A$ is $I$-adically complete,
    then $A/I$ is a quotient of a Gorenstein ring if and only if $A$ is. 
\end{Thm}

On the other hand, quotients of Cohen-Macaulay rings were also studied by Kawasaki \cite{Kawasaki-quotient-of-Macaulay}.
Note that the conditions (C1)-(C3) there are equivalent to CM-excellence as in \cite[Definition 1.2]{Macaulay-Cesnavi} (cf. \cite[Remark 1.5]{Macaulay-Cesnavi}).
Therefore \cite[Theorem 1.3]{Kawasaki-quotient-of-Macaulay} and its proof tell us the following.
\begin{Thm}\label{thm:QuotCMiffs}
    Let $A$ be a Noetherian ring.
    Then the followings are equivalent.
    \begin{enumerate}[label=$(\roman*)$]
        \item $A$ is a quotient of a Cohen-Macaulay ring.
        \item There exists a finitely generated $A$-algebra $B$ that is Cohen-Macaulay and admits a section $B\to A$.
        \item $A$ is CM-excellent and admits a codimension function.
    \end{enumerate}
\end{Thm}

Pham Hung Quy asks the author if the analog of Theorem \ref{thm:QuotGorLifts} holds for Cohen-Macaulay instead of Gorenstein.
We can answer this question affirmatively under some additional assumptions.
\begin{Thm}\label{thm:QuotCMLifts}
    Let $A$ be a Noetherian ring, $I$ an ideal of $A$.
    Assume that $A/I$ is either quasi-excellent, or semilocal and Nagata.
    If $A$ is $I$-adically complete,
    then $A/I$ is a quotient of a Cohen-Macaulay ring if and only if $A$ is. 
\end{Thm}

\begin{Lem}\label{lem:LiftingCodimFunction}
    Let $A$ be a Noetherian ring, $I$ an ideal of $A$.
    Assume that the pair $(A,I)$ is Henselian.

    Assume that $A$ is catenary, and that $A/I$ admits a codimension function $c:\Spec(A/I)\to \bZ$.
    For $\fp\in\Spec(A)$ and $Q\in V(I)$ containing $\fp$,
    let $c(\fp,Q)=c(Q/I)-\HT(Q/\fp)$.
    Then the followings hold.
    \begin{enumerate}[label=$(\roman*)$]
        \item\label{Liftcodim:cWDF} For every $\fp\in\Spec(A)$,
    $c(\fp,Q)$ is independent of the choice of $Q$.
    \item\label{Liftcodim:ciscodim}  The assignment $\fp\mapsto c(\fp,Q)$,
    where $Q\in V(I+\fp)$ is arbitrary,
    is a codimension function on $A$.
    \end{enumerate}
\end{Lem} 
\begin{proof}
    Since $c(-)$ is a codimension function of $A/I$, for $Q\subseteq Q'\in V(I+\fp)$, we have
    $c(\fp,Q')-c(\fp,Q)=\HT(Q'/Q)+\HT(Q/\fp)-\HT(Q'/\fp)$.
    Since $A$ is catenary, the number on the right hand side is $0$.
    Thus for $Q\subseteq Q'\in V(I+\fp)$, we have
    $c(\fp,Q')=c(\fp,Q)$.
    Since $(A,I)$ is Henselian, $V(I+\fp)$ is connected,
    see \citestacks{09Y6}.
    This shows \ref{Liftcodim:cWDF}.
    For \ref{Liftcodim:ciscodim}, note that $V(I+\fp)$ is nonempty for all $\fp\in\Spec(A)$, so our function is well-defined.
    To show it is a codimension function,
    it suffices to show for $\fp\subseteq\fp'\subseteq Q$ where $Q\in V(I)$,
    we have $c(\fp',Q)-c(\fp,Q)=\HT(\fp'/\fp)$.
    This is clear as $A$ is catenary.
\end{proof}

\begin{Lem}\label{lem:liftUCcodim}
    Let $A$ be a Noetherian ring, $I$ an ideal of $A$.
    Assume that the pair $(A,I)$ is Henselian.

    Assume the followings hold.
    \begin{enumerate}[label=$(\roman*)$]
        \item\label{liftUC:fintoUC} For every minimal prime $\fp$ of $A$,
        there exists a finite injective ring map $A/\fp\to B$ such that $B$ is universally catenary.
        \item\label{liftUC:quotisUC} $A/I$ is universally catenary and admits a codimension function.
    \end{enumerate}

    Then $A$ is universally catenary and admits a codimension function.
\end{Lem}
\begin{proof}
    By Lemma \ref{lem:LiftingCodimFunction},
    it suffices to show $A$ is universally catenary.
    We may therefore assume $A$ is an integral domain and there exists a finite injective ring map $A\to B$ such that $B$ is universally catenary.
    We may also assume $B$ is an integral domain.
    Let $c:\Spec(A/I)\to\bZ$ be a codimension function.
    Let $\delta(Q)=c((Q\cap A)/I)-\HT(Q)$ for $Q\in V(IB)$.
    For $Q\subseteq Q'\in V(IB)$,
    we have 
    \begin{align*}
        c((Q'\cap A)/I)-c((Q\cap A)/I)&=\HT((Q'\cap A)/(Q\cap A))\\
        &=\HT(Q'/Q)\\
        &=\HT(Q')-\HT(Q)
    \end{align*}
    where the first identity is because $c$ is a codimension function,
    the second identity follows from the dimension formula \citestacks{02IJ} as $A/I$ is universally catenary,
    and the third identity is because $B$ is a catenary domain.
    Therefore $\delta(Q')=\delta(Q)$.
    Since the pair $(B,IB)$ is Henselian \citestacks{09XK},
    $V(IB)$ is connected, see \citestacks{09Y6}.
    Thus $\delta$ is a constant function on $V(IB)$.
    Therefore, for all maximal ideals $\fn$ of $B$,
    which necessarily contains $IB$, $\HT(\fn)$ depends only on $\fn\cap A$,
    and thus must equal to $\HT(\fn\cap A)$.
    As discussed in Remark \ref{rem:etaleUC},
    we see $A$ is universally catenary,
    as desired.
\end{proof}
\begin{proof}[Proof of Theorem \ref{thm:QuotCMLifts}]
    Assume that $A$ is $I$-adically complete, $A/I$ is a quotient of a Cohen-Macaulay ring, and $A/I$ is either quasi-excellent or semilocal and Nagata.
    Note that $(A,I)$ is a Henselian pair \citestacks{0ALJ}.

    If $A/I$ is quasi-excellent, then $A$ is quasi-excellent \cite{formal-lifting-excellence-Gabber},
    in particular CM-quasi-excellent.
    If $A/I$ is semilocal and Nagata,
    then $A$ is semilocal and Nagata \cite{Marot-Nagata-Lift},
    and thus has Cohen-Macaulay formal fibers by \cite[Theorem C]{Mur-Grothendieck-localization}.
    Then $A$ is CM-quasi-excellent, see \cite[Proposition 7.3.18]{EGA4_2}.
    Consequently, in both cases, $A$ is CM-quasi-excellent.
    
    Now $A$ satisfies the condition \ref{liftUC:fintoUC} in Lemma \ref{lem:liftUCcodim} by the discussions in Remark \ref{rem:etaleUC},
    thus is universally catenary and admits a codimension function by Lemma \ref{lem:liftUCcodim}.
    By Theorem \ref{thm:QuotCMiffs},
    $A$ is a quotient of a Cohen-Macaulay ring.
\end{proof}
\printbibliography
\end{document}